\documentclass[12pt,a4paper, reqno]{article}
\usepackage[utf8]{inputenc}
\usepackage[T1]{fontenc}
\usepackage{float}
\usepackage[scale=0.82]{geometry}
\usepackage[english]{babel}
\usepackage{calc}
\usepackage{enumitem}
\usepackage{booktabs}
\usepackage{amsmath}
\usepackage{amsthm}
\usepackage{amssymb}
\usepackage{xcolor}
\usepackage{graphicx}
\usepackage{xurl}
\usepackage{array,graphicx}
\usepackage{booktabs}
\usepackage{multirow}
\usepackage[labelformat=simple]{subcaption}

\usepackage[export]{adjustbox}
\definecolor{webgreen}{rgb}{0,.5,0}
\definecolor{webbrown}{rgb}{.6,0,0}
\usepackage[colorlinks=true,
linkcolor=webgreen,
filecolor=webbrown,
citecolor=webgreen]{hyperref}
\usepackage{tikz}

\newtheorem{theorem}{Theorem}
\newtheorem{lemma}[theorem]{Lemma}
\theoremstyle{remark}
\newtheorem{remark}[theorem]{Remark}

\newcommand{\seqnum}[1]{\href{http://oeis.org/#1}{\underline{#1}}}
\everymath{\displaystyle}

\newcommand{\CK}{\mathit{CK}}

\newcommand{\goi}{\protect\includegraphics[height=1.4cm,valign=c]{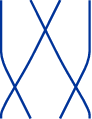}}
\newcommand{\goii}{\protect\includegraphics[height=1.4cm,valign=c]{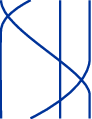}}
\newcommand{\goiii}{\protect\includegraphics[height=1.4cm,valign=c]{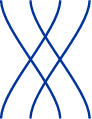}}

\newcommand{\gi}{\protect\includegraphics[height=1.4cm,valign=c]{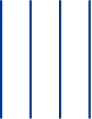}}
\newcommand{\gii}{\protect\includegraphics[height=1.4cm,valign=c]{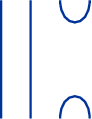}}
\newcommand{\giii}{\protect\includegraphics[height=1.4cm,valign=c]{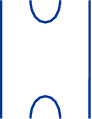}}
\newcommand{\giv}{\protect\includegraphics[height=1.4cm,valign=c]{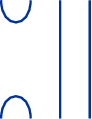}}
\newcommand{\gv}{\protect\includegraphics[height=1.4cm,valign=c]{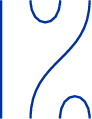}}
\newcommand{\gvi}{\protect\includegraphics[height=1.4cm,valign=c]{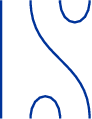}}
\newcommand{\gvii}{\protect\includegraphics[height=1.4cm,valign=c]{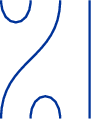}}
\newcommand{\gviii}{\protect\includegraphics[height=1.4cm,valign=c]{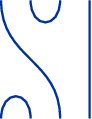}}
\newcommand{\gix}{\protect\includegraphics[height=1.4cm,valign=c]{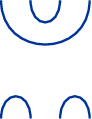}}
\newcommand{\gx}{\protect\includegraphics[height=1.4cm,valign=c]{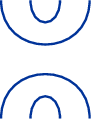}}
\newcommand{\gxi}{\protect\includegraphics[height=1.4cm,valign=c]{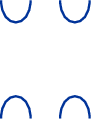}}
\newcommand{\gxii}{\protect\includegraphics[height=1.4cm,valign=c]{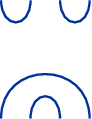}}
\newcommand{\gxiii}{\protect\includegraphics[height=1.4cm,valign=c]{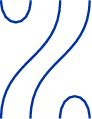}}
\newcommand{\gxiv}{\protect\includegraphics[height=1.4cm,valign=c]{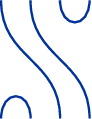}}

\newcommand{\gicII}{\protect\includegraphics[width=2cm,valign=c]{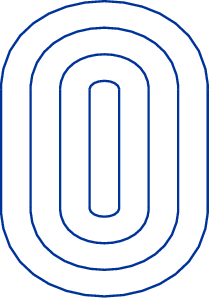}}
\newcommand{\giicII}{\protect\includegraphics[width=2cm,valign=c]{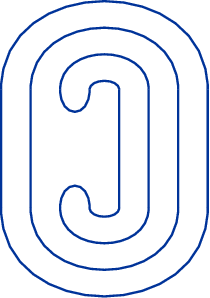}}
\newcommand{\giiicII}{\protect\includegraphics[width=2cm,valign=c]{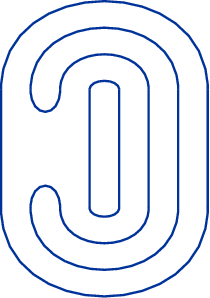}}
\newcommand{\givcII}{\protect\includegraphics[width=2cm,valign=c]{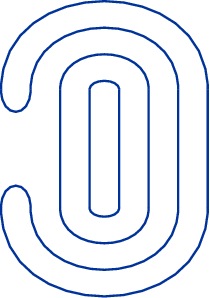}}
\newcommand{\gvcII}{\protect\includegraphics[width=2cm,valign=c]{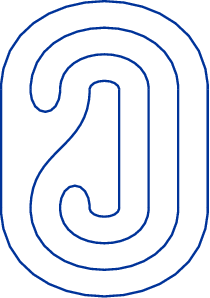}}
\newcommand{\gvicII}{\protect\includegraphics[width=2cm,valign=c]{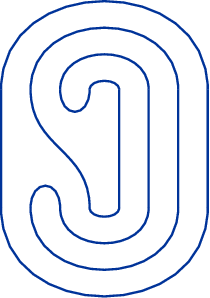}}
\newcommand{\gviicII}{\protect\includegraphics[width=2cm,valign=c]{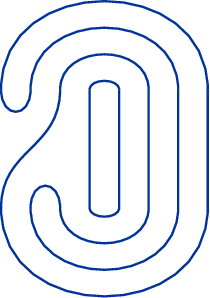}}
\newcommand{\gviiicII}{\protect\includegraphics[width=2cm,valign=c]{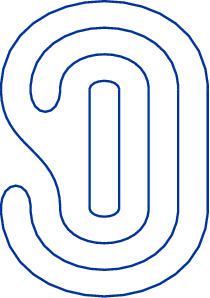}}
\newcommand{\gixcII}{\protect\includegraphics[width=2cm,valign=c]{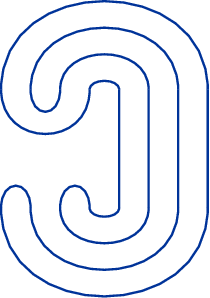}}
\newcommand{\gxcII}{\protect\includegraphics[width=2cm,valign=c]{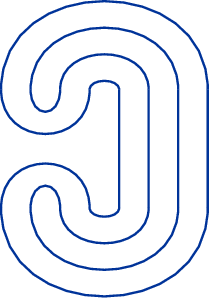}}
\newcommand{\gxicII}{\protect\includegraphics[width=2cm,valign=c]{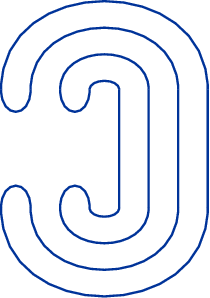}}
\newcommand{\gxiicII}{\protect\includegraphics[width=2cm,valign=c]{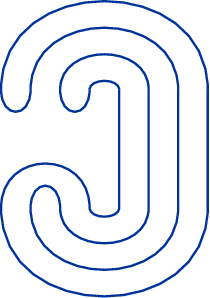}}
\newcommand{\gxiiicII}{\protect\includegraphics[width=2cm,valign=c]{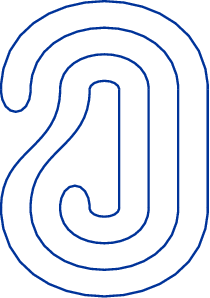}}
\newcommand{\gxivcII}{\protect\includegraphics[width=2cm,valign=c]{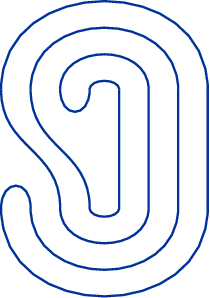}}

\makeatletter
\def\namedlabel#1#2{\begingroup
	#2%
	\def\@currentlabel{#2}%
	\phantomsection\label{#1}\endgroup
}
\makeatother

\title{\bf A note on the recursive computation of the bracket polynomial for closed 4-tangles}
\author{Franck Ramaharo\\
	\small Département de  Mathématiques et Informatique\\[-0.8ex]
	\small Université d'Antananarivo\\[-0.8ex] 
	\small 101 Antananarivo, Madagascar\\
	\small\href{mailto:franck.ramaharo@gmail.com}{\tt franck.ramaharo@gmail.com}\\
}
\date{\small August 25, 2025\\}

\newcommand{\brk}[1]{\left<#1\right>}
\begin{document}
	\maketitle
	
	\begin{abstract}
		Given a 4-tangle shadow, we concatenate it with itself $n$ times and form a knot by applying a closure operation that connects each top endpoint to the corresponding bottom endpoint on the same side without introducing any crossings. We then compute the Kauffman bracket polynomial for the resulting knot using a states matrix defined with respect to the basis of the Kauffman 4-strand diagram monoid.	
		
	\end{abstract}
	
	In this note, we compute the bracket polynomial for the closure of an iterated 4-tangle shadow, $T_n = TT\cdots T$ (that is, $T$ concatenated with itself $n$ times), via the states matrix. Given a tangle shadow $T$ (from this point onward, all tangles are understood to be shadow diagrams, and we will simply refer to them as tangles), its bracket polynomial is a linear combination of the form 
	\begin{equation}\label{eq:bracket1}
		\brk{T} = \brk{\protect\includegraphics[height=1.5cm,valign=c]{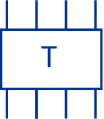}} = \sum_{i=1}^{14} a_i \brk{g_i},
	\end{equation}
	where $g_i$ are the element basis of the monoid diagram $\mathcal{K}_4$ \cite{KitovNikov2020} as shown in Figure~\ref{fig:basisstates}, $a_i\in \mathbb{Z}[x]$, and the bracket polynomial $\brk{\ \cdot\ }$ is defined by the following axioms
	\begin{itemize}
		\item [\namedlabel{itm:K1}{$ (\mathbf{K1}) $}:]  $\left<\bigcirc \right>=x $;
		\item [\namedlabel{itm:K2}{$ (\mathbf{K2}) $}:]  $\left<\bigcirc\sqcup D'\right>=x\left< D'\right>$;
		\item [\namedlabel{itm:K3}{$ (\mathbf{K3}) $}:]  $\left<\protect\includegraphics[width=.0225\linewidth,valign=c]{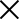}\right>=\left<\protect\includegraphics[width=.0225\linewidth,valign=c]{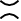}\right>+\left<\protect\includegraphics[width=.0225\linewidth,valign=c]{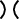}\right>$.
	\end{itemize}
	\begin{figure}[!ht]
		\centering
		\makebox[0pt]{\begin{tabular}{wc{2cm}wc{2cm}wc{2cm}wc{2cm}wc{2cm}wc{2cm}wc{2cm}}
				\gi & \gii & \giii & \giv & \gv & \gvi & \gvii\\
				$g_1$ & $g_2$ &$g_3$ &$g_4$ &$g_5$ &$g_6$ &$g_7$\\[3ex]
				\gviii & \gix & \gx & \gxi & \gxii & \gxiii & \gxiv\\
				$g_8$ & $g_9$ &$g_{10}$ &$g_{11}$ &$g_{12}$ &$g_{13}$ &$g_{14}$
		\end{tabular}}
		\caption{The states basis element of $\mathcal{K}_4$}
		\label{fig:basisstates}
	\end{figure}
	
	The concatenation of two 4-tangles $T$ and $T'$, denoted $T T'$, is the 4-tangle formed by placing $T$  directly above $T'$  and connecting the bottom endpoints of $T$  to the corresponding top endpoints of $T'$  with non-crossing planar arcs, see Figure~\ref{fig:concatenation}. We are particularly interested in the iteration of a generator tangle $T$, denoted $T_n := TT\cdots T$ ($n$ times). If $a_i^{(n)}$ denote the coefficients of the  $n$-th iteration, $i=1,\ldots,14$, then the associated bracket polynomial is expressed as \begin{equation}\label{eq:bracketn}
		\brk{T_n}:=\sum_{i=1}^{14} a_i^{(n)}\left<g_i\right>.
	\end{equation}
	
	For a tangle $T$, the closure for $T$ under consideration in this note, denoted $\overline{T}$, is formed by connecting each top endpoint to the corresponding bottom endpoint on the same side  without introducing any crossings as depicted in Figure~\ref{fig:tangleclosure}.
	
	\begin{figure}[H]
		\centering
		\makebox[0pt]{\begin{subfigure}[b]{0.6\textwidth}
				\centering
				$T:=\protect\includegraphics[height=1.5cm,valign=c]{tangleT},\qquad T':=\protect\includegraphics[height=1.5cm,valign=c]{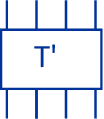},\quad  TT' = \protect\includegraphics[height=3cm,valign=c]{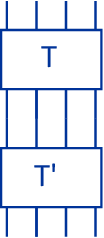}$
				\caption{Concatenation of two $4$-tangles}
				\label{fig:concatenation}
			\end{subfigure}%
			\begin{subfigure}[b]{0.4\textwidth}
				\centering
				\includegraphics[height=4cm]{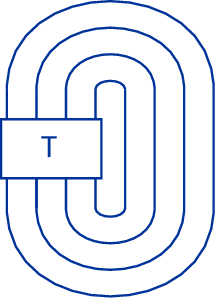}
				\caption{Closure of a $4$-tangle}
				\label{fig:tangleclosure}
			\end{subfigure}%
		}
		\caption{Operations on $4$-tangles}
	\end{figure}
	
	\begin{lemma}\label{lem:closure}
		The bracket polynomial for the closure of $T$ is given by:\begin{equation}
			\left<\overline{T}\right> = a_{1}x^4 +  \left(a_2 + a_3 + a_4 \right)x^3 + (a_5 + a_6 + a_7 + a_{8} + a_{10} + a_{11})x^2 +(a_{9} + a_{12} + a_{13} + a_{14})x.
		\end{equation}
		
	\end{lemma}
	\begin{figure}[H]
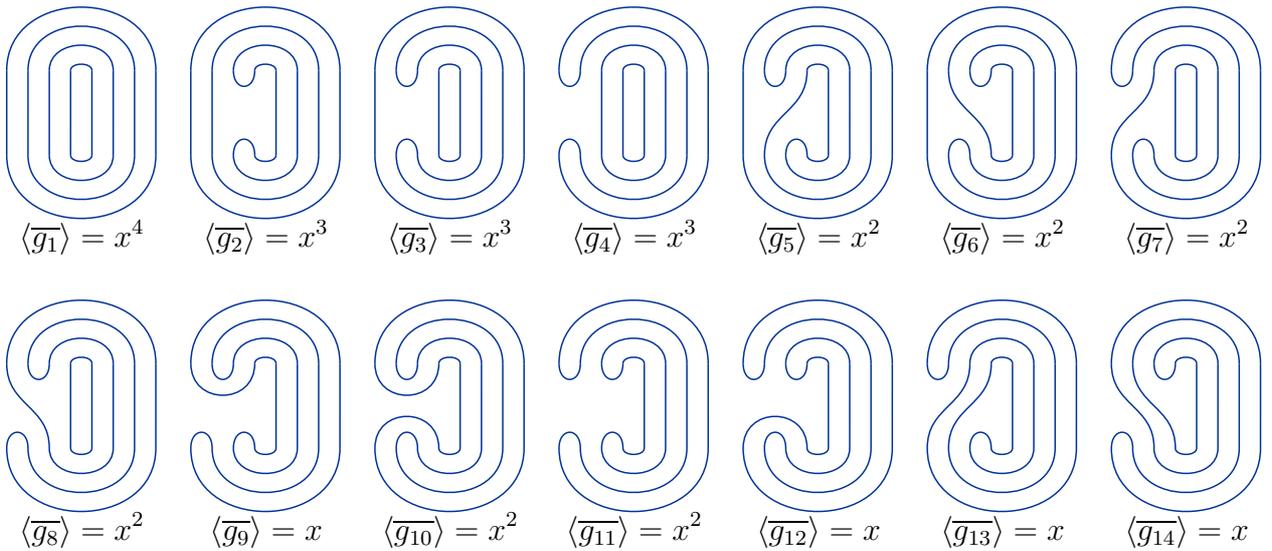

		\centering
		\makebox[0pt]{\begin{tabular}{ccccccc}
				\gicII                       &  \giicII                      &    \giiicII                     &  \givcII                       & \gvcII                       &  \gvicII                        &    \gviicII \\[5ex]
				$\brk{\overline{g_1}}=x^4$   &  $\brk{\overline{g_2}}=x^3$   &	$\brk{\overline{g_3}}=x^3$     &   $\brk{\overline{g_4}}=x^3$   & $\brk{\overline{g_5}}=x^2$   &  $\brk{\overline{g_6}}=x^2$     &	$\brk{\overline{g_7}}=x^2$ \\[3ex]
				\gviiicII                    &  \gixcII                      &    \gxcII                       &  \gxicII                       & \gxiicII                     &  \gxiiicII                      &   \gxivcII  \\[3ex]
				$\brk{\overline{g_{8}}}=x^2$ &  $\brk{\overline{g_9}}=x$     &	$\brk{\overline{g_{10}}}=x^2$  &  $\brk{\overline{g_{11}}}=x^2$ & $\brk{\overline{g_{12}}}=x$  &   $\brk{\overline{g_{13}}}=x$   &	$\brk{\overline{g_{14}}}=x$  \\
		\end{tabular}}
		\caption{Closures of the elements of $\mathcal{K}_4$}
		\label{fig:Dclosure}
	\end{figure}
	
	\begin{proof}
		The bracket polynomial  for the closure of $T$ is evaluated on the elements of $\mathcal{K}_4$, i.e.,
		\begin{equation}\label{eq:closuredefinitionD}
			\left<\overline{T}\right> = \sum_{i=1}^{14}a_i\left<\overline{g_i}\right>.
		\end{equation}
		Values of the $\brk{\overline{g_i}}$'s are given in Figure~\ref{fig:Dclosure}.
	\end{proof}

	\begin{lemma}\label{lem:concatbracket}
		If $T$ and $T'$ are two 4-tangles such that $\brk{T} = \sum_{i=1}^{14}a_i\brk{g_i}$ and $\brk{T'} = \sum_{i=1}^{14}b_i\brk{g_i}$, then $\brk{TT'}$ satisfies
		\begingroup
		\allowdisplaybreaks
		\begin{align*}
			\brk{TT'} & =   b_{1} a_{1}\brk{g_1}\\
			& + \left(b_{2} a_{1}  + \left(b_{1}  + b_{5} + b_{2} x\right) a_{2} + \left(b_{2} + b_{13} + b_{5} x\right) a_{6} + \left(b_{5} + b_{13} x\right)  a_{14}\right) \brk{g_{2}}\\
			& + \left(b_{3} a_{1}  + \left(b_{1}  + b_{6} + b_{7} + b_{3} x\right) a_{3} +\left(b_{3} + b_{6} x\right) a_{5}   + \left(b_{3} + b_{7} x\right) a_{8}\right) \brk{g_{3}}\\
			& + \left(b_{4} a_{1}  + \left(b_{1}  + b_{8} + b_{4} x\right) a_{4} + \left(b_{4} + b_{14} + b_{8} x\right) a_{7} + \left(b_{8} + b_{14} x\right) a_{13}\right) \brk{g_{4}}\\ 
			& + \left(b_{5} a_{1}  + \left(b_{2}  + b_{13} + b_{5} x\right) a_{3} + \left(b_{1} + b_{5} + b_{2} x\right) a_{5} + \left(b_{5} + b_{13} x\right) a_{8}\right) \brk{g_{5}} \\
			& + \left(b_{6} a_{1}  + \left(b_{3}  + b_{6} x\right)  a_{2} + \left(b_{1} + b_{6} + b_{7} + b_{3} x\right) a_{6} + \left(b_{3} + b_{7} x\right) a_{14} \right) \brk{g_{6}} \\
			& + \left(b_{7} a_{1}  + \left(b_{3}  + b_{7} x\right)  a_{4} + \left(b_{1} + b_{6} + b_{7} + b_{3} x\right) a_{7} + \left(b_{3} + b_{6} x\right) a_{13} \right) \brk{g_{7}}\\
			& + \left(b_{8} a_{1}  + \left(b_{4}  + b_{14} + b_{8} x\right)  a_{3}   + \left(b_{8} + b_{14} x\right) a_{5}     + \left(b_{1} + b_{8} + b_{4} x\right) a_{8}\right) \brk{g_{8}}\\
			& + \left(b_{9} a_{1}  + \left(b_{11} + b_{9} x\right)  a_{3}            + \left(b_{4} + b_{9} + b_{13} + b_{11} x\right)  a_{5} + \left(b_{2} + b_{9} + b_{14} + b_{11} x\right) a_{8} \right.\\
			& \quad +       \left(b_{1} + b_{5} + b_{8}  + \left(b_{2} + b_{4} + b_{9} + b_{13} + b_{14}\right) x + b_{11} x^2\right) a_{9} \\
			& \quad + \left.\left(b_{2} + b_{4} + b_{13} + b_{14} + \left(b_{5} + b_{8} + b_{11}\right) x         + b_{9} x^2\right) a_{10} \right) \brk{g_{9}} \tag{\theequation}\label{eq:Tn1}\\
			& + \left(b_{10} a_{1}  + \left(b_{12} + b_{10} x\right) a_{3}           + \left(b_{7} + b_{10} + b_{12} x\right) a_{5}   + \left(b_{6} + b_{10} + b_{12} x\right) a_{8} \right.     \\
			& \quad   + \left. \left(b_{3} + \left(b_{6} + b_{7} + b_{10}\right) x + b_{12} x^2\right) a_{9}                          + \left(b_{1} + b_{6} + b_{7} + \left(b_{3} + b_{12}\right) x + b_{10} x^2\right) a_{10} \right) \brk{g_{10}}\\ 
			& + \left(b_{11} a_{1} + \left(b_{4} + b_{9} + b_{13} + b_{11} x\right) a_{2}  + \left(b_{2} + b_{9} + b_{14} + b_{11} x\right) a_{4}    + \left(b_{11} + b_{9} x\right) a_{6} \right.\\
			& \quad +       \left(b_{11} + b_{9} x\right)  a_{7}  + \left(b_{1} + b_{5} + b_{8}   + \left(b_{2} + b_{4} + b_{9}  + b_{13} + b_{14}\right) x + b_{11} x^2\right) a_{11} \\
			& \quad +       \left(b_{2}  + b_{4} + b_{13} + b_{14} + \left(b_{5} + b_{8} + b_{11}\right) x +  b_{9} x^2\right)  a_{12} + \left(b_{4} + b_{9} + b_{13} + b_{11} x\right) a_{13} \\
			& \quad + \left.\left(b_{2} + b_{9} + b_{14} + b_{11} x\right) a_{14} \right) \brk{g_{11}}\\
			& + \left(b_{12} a_{1} + \left(b_{7} + b_{10} + b_{12} x\right) a_{2}      + \left(b_{6} + b_{10} + b_{12} x\right)  a_{4}     + \left(b_{12} + b_{10} x\right)  a_{6}          +  \left(b_{12} + b_{10} x\right) a_{7}\right.\\
			& \quad +       \left(b_{3} + \left(b_{6} + b_{7} + b_{10}\right) x +  b_{12} x^2\right)  a_{11}                               + \left(b_{1} + b_{6} + b_{7} + \left(b_{3} + b_{12}\right) x + b_{10} x^2\right) a_{12}\\
			& \quad + \left.\left(b_{7} + b_{10} + b_{12} x\right) a_{13} + \left(b_{6} + b_{10} + b_{12} x\right) a_{14} \right) \brk{g_{12}} \\
			& + \left(b_{13} a_{1} + \left(b_{5} + b_{13} x\right) a_{4}  + \left(b_{2} + b_{13} + b_{5} x\right) a_{7} + \left(b_{1} + b_{5} + b_{2} x\right) a_{13} \right) \brk{g_{13}}\\ 
			& + \left(b_{14} a_{1} + \left(b_{8} + b_{14} x\right) a_{2}  + \left(b_{4} + b_{14} + b_{8} x\right) a_{6} + \left(b_{1} + b_{8} + b_{4} x\right) a_{14} \right) \brk{g_{14}}.
		\end{align*}
		\endgroup
	\end{lemma}
	
	\begin{proof}
		We first establish the states of $ T $ leaving $ T' $ intact, and then in $ T' $:
		\begin{equation}
			\brk{TT'} = \sum_{i=1}^{14}a_i\sum_{j=1}^{14}b_j \brk{g_i g_j}.
		\end{equation}
		The brackets for the pairs $(g_i, g_j)$ can be evaluated by applying the concatenation table in Table~\ref{tab:concattable}. 
		
	\end{proof}
	
	\begin{table}
		\centering
		\makebox[0pt]{\includegraphics[width=1.05\linewidth]{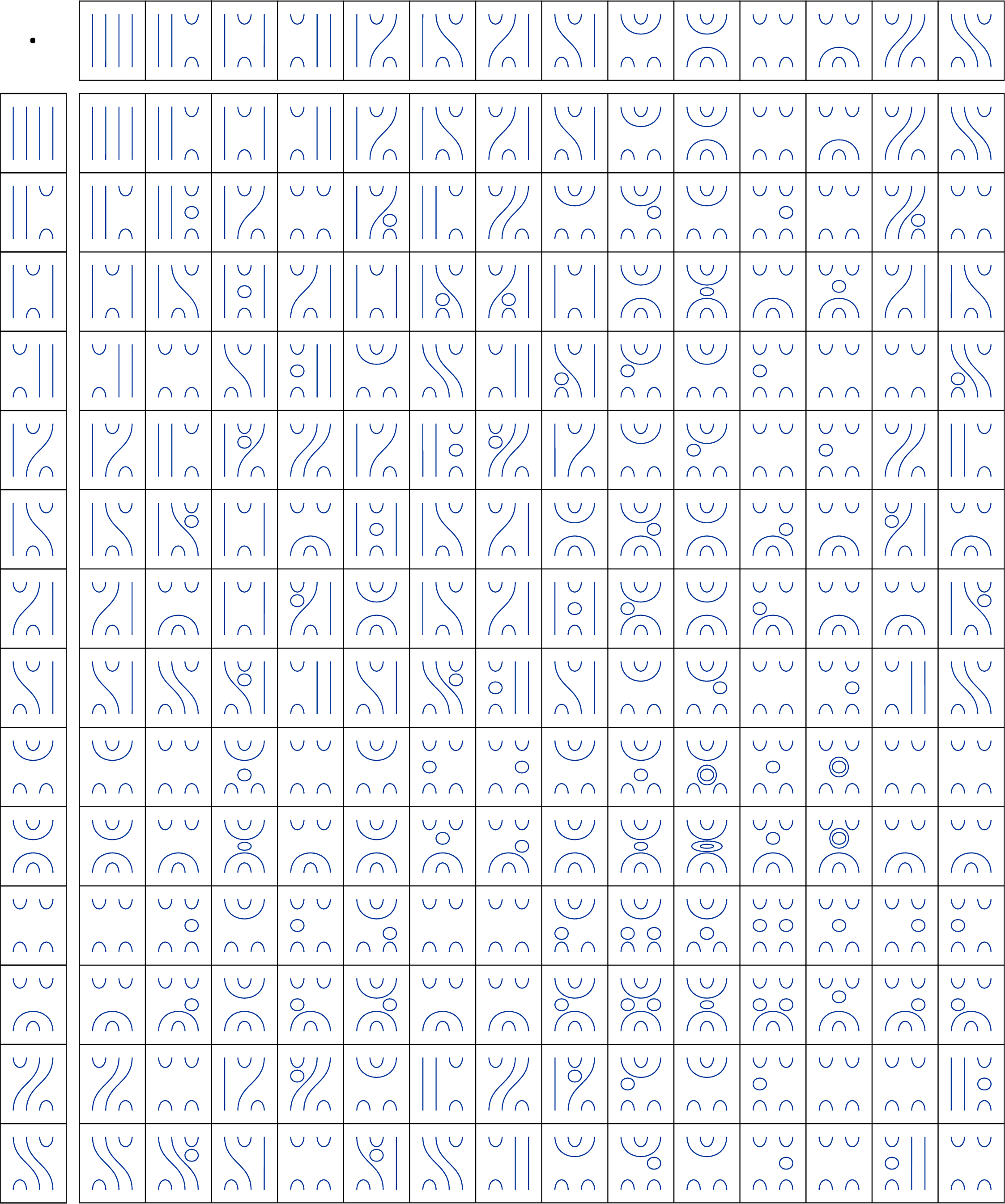}}
		\caption{Concatenation table of $g_i$ and $g_{j}$. Entry in row $j$ and column $i$ gives the resulting state  from the concatenation $g_ig_{j}$, with $i=1,\ldots,14$  indexing the column labels (top row) and $j=1,\ldots 14$  indexing the column labels (leftmost column).}
		\label{tab:concattable}
	\end{table}
	
	For convenience, we identify the bracket polynomial expression in \eqref{eq:bracket1}, and similarly that of \eqref{eq:bracketn}, with the bracket vectors
	\[
	A:=A(T) = \left[ a_1, a_2, a_3, a_4, a_5, a_6, a_7, a_8, a_9, a_{10}, a_{11}, a_{12}, a_{13}, a_{14} \right]^\intercal
	\]
	and
	\[
	A_n:=A(T_n) = \left[ a_1^{(n)}, a_2^{(n)}, a_3^{(n)}, a_4^{(n)}, a_5^{(n)}, a_6^{(n)}, a_7^{(n)}, a_8^{(n)}, a_9^{(n)}, a_{10}^{(n)}, a_{11}^{(n)}, a_{12}^{(n)}, a_{13}^{(n)}, a_{14}^{(n)} \right]^\intercal,
	\] respectively.

	\begin{lemma} For $n\geq 1$,  $A_n$ satisfies the following recursion:		
		
		\begin{equation}\label{eq:matrixformulation}
			\renewcommand\arraystretch{1.325}
			\scriptstyle
			\begin{pmatrix}
				a_1^{(n)}\\ 
				a_2^{(n)}\\ 
				a_3^{(n)}\\ 
				a_4^{(n)}\\ 
				a_5^{(n)}\\ 
				a_6^{(n)}\\ 
				a_7^{(n)}\\ 
				a_8^{(n)}\\ 
				a_9^{(n)}\\ 
				a_{10}^{(n)}\\ 
				a_{11}^{(n)}\\ 
				a_{12}^{(n)}\\ 
				a_{13}^{(n)}\\ 
				a_{14}^{(n)}
			\end{pmatrix}=
			\begin{pmatrix}
				a_{1}  & 0      & 0      & 0      & 0      & 0      & 0      & 0      & 0      & 0      & 0      & 0      & 0      & 0 \\
				a_{2}  & u_{1}  & 0      & 0      & 0      & u_{12} & 0      & 0      & 0      & 0      & 0      & 0      & 0      & u_{6} \\
				a_{3}  & 0      & u_{10} & 0      & u_{2}  & 0      & 0      & u_{4}  & 0      & 0      & 0      & 0      & 0      & 0 \\
				a_{4}  & 0      & 0      & u_{13} & 0      & 0      & u_{8}  & 0      & 0      & 0      & 0      & 0      & u_{3}  & 0 \\
				a_{5}  & 0      & u_{12} & 0      & u_{1}  & 0      & 0      & u_{6}  & 0      & 0      & 0      & 0      & 0      & 0 \\
				a_{6}  & u_{2}  & 0      & 0      & 0      & u_{10} & 0      & 0      & 0      & 0      & 0      & 0      & 0      & u_{4} \\
				a_{7}  & 0      & 0      & u_{4}  & 0      & 0      & u_{10} & 0      & 0      & 0      & 0      & 0      & u_{2}  & 0 \\
				a_{8}  & 0      & u_{8}  & 0      & u_{3}  & 0      & 0      & u_{13} & 0      & 0      & 0      & 0      & 0      & 0 \\
				a_{9}  & 0      & u_{5}  & 0      & u_{11} & 0      & 0      & u_{16} & u_{18} & u_{19} & 0      & 0      & 0      & 0 \\
				a_{10} & 0      & u_{9}  & 0      & u_{14} & 0      & 0      & u_{15} & u_{17} & u_{20} & 0      & 0      & 0      & 0 \\
				a_{11} & u_{11} & 0      & u_{16} & 0      & u_{5}  & u_{5}  & 0      & 0      & 0      & u_{18} & u_{19} & u_{11} & u_{16} \\
				a_{12} & u_{14} & 0      & u_{15} & 0      & u_{9}  & u_{9}  & 0      & 0      & 0      & u_{17} & u_{20} & u_{14} & u_{15} \\
				a_{13} & 0      & 0      & u_{6}  & 0      & 0      & u_{12} & 0      & 0      & 0      & 0      & 0      & u_{1}  & 0 \\
				a_{14} & u_{3}  & 0      & 0      & 0      & u_{8}  & 0      & 0      & 0      & 0      & 0      & 0      & 0      & u_{13}
			\end{pmatrix}\begin{pmatrix}
				a_1^{(n-1)}\\ 
				a_2^{(n-1)}\\ 
				a_3^{(n-1)}\\ 
				a_4^{(n-1)}\\ 
				a_5^{(n-1)}\\ 
				a_6^{(n-1)}\\ 
				a_7^{(n-1)}\\ 
				a_8^{(n-1)}\\ 
				a_9^{(n-1)}\\ 
				a_{10}^{(n-1)}\\ 
				a_{11}^{(n-1)}\\ 
				a_{12}^{(n-1)}\\ 
				a_{13}^{(n-1)}\\ 
				a_{14}^{(n-1)}
			\end{pmatrix}
		\end{equation}
		where \[u_i:=u_i(a_1, a_2, a_3, a_4, a_5, a_6, a_7, a_8, a_9, a_{10}, a_{11}, a_{12}, a_{13}, a_{14}),\ i=1,\ldots,19,\] with
		\[\begin{array}{@{}wl{6.5cm}wl{5.cm}wl{5.cm}@{}}
			u_1 := a_{1} + a_{5} + a_{2} x,              & u_2 := a_{3} + a_{6} x,              & u_3 := a_{8} + a_{14} x, \\
			u_4 := a_{3} + a_{7} x,                      & u_5 := a_{11} + a_{9} x,             & u_6 := a_{5} + a_{13} x,\\
			u_7 := a_{14} + a_{4} + a_{8} x,             & u_8 := a_{12} + a_{10} x,            & u_9 := a_{1} + a_{6} + a_{7} + a_{3} x,\\
			u_{10} := a_{13} + a_{4} + a_{9} + a_{11} x, & u_{11} := a_{13} + a_{2} + a_{5} x,  & u_{12} := a_{1} + a_{8} + a_{4} x, \\ 
			u_{13} := a_{10} + a_{7} + a_{12} x,         & u_{14} := a_{10} + a_{6} + a_{12} x, & u_{15} := a_{14} + a_{2} + a_{9} + a_{11} x, \\ 
			\multicolumn{2}{@{}l}{u_{16} := a_{3} + (a_{10} + a_{6} + a_{7}) x + a_{12} x^2,}\\
			\multicolumn{2}{@{}l}{u_{17} := a_{1} + a_{5} + a_{8} +( a_{13} + a_{14} + a_{2} + a_{4} + a_{9}) x + a_{11} x^2,} \\
			\multicolumn{2}{@{}l}{u_{18} := a_{13} + a_{14} + a_{2} + a_{4} + (a_{11} + a_{5} + a_{8}) x + a_{9} x^2,}\\
			\multicolumn{2}{@{}l}{u_{19} := a_{1} + a_{6} + a_{7} + (a_{12} + a_{3}) x + a_{10} x^2.}
		\end{array}\]		
	\end{lemma}
	
	\begin{proof}
		We write
		\[\brk{T_n} = \brk{T_{n-1} T}=\sum_{i=1}^{14}a_i^{(n-1)}\sum_{j=1}^{14} a_j \brk{g_i g_j},\] then we apply Lemma~\ref{lem:concatbracket}. The matrix expression follows from \eqref{eq:Tn1} by setting $T=T'$.
	\end{proof}
	Let $M = M(a_1, a_2, a_3, a_4, a_5, a_6, a_7, a_8, a_9, a_{10}, a_{11}, a_{12}, a_{13}, a_{14})$ denote the $14\times14$ matrix in \eqref{eq:matrixformulation}, and be referred to as the \textit{states matrix}. Then \eqref{eq:matrixformulation} also reads
	\begin{equation}\label{eq:matrixform}
		A_{n} = M A_{n-1} = M^n A_0,
	\end{equation} with $$A_0 = \left[1,0,0,0,0,0,0,0,0,0,0,0,0,0\right]^\intercal.$$
	
	\begin{theorem}
		Let $\overline{X}:=\left[x^4, x^3, x^3, x^3, x^2, x^2, x^2, x^2, x, x^2, x^2, x, x, x\right]^\intercal$. Then 
		the bracket polynomial for the closure of $T_n$  is given by
		\begin{equation}
			\left<\overline{T_n}\right> = A_n\cdot \overline{X},
		\end{equation}
		where $\cdot$ denotes the dot product.	
	\end{theorem}

	\begin{proof}
		This is a consequence of Lemma~\ref{lem:closure} and formula \eqref{eq:matrixform}.
	\end{proof}

	As an application, we compute bracket polynomials of the three closed 4-strand braided designs illustrated in Figure~\ref{fig:4braid}: the 4-strand square Turk's head knot \cite[\#1322]{Ashley1994}, the 4-strand plate sinnet \cite[\#2959]{Ashley1994}, and the 4-strand flat sinnet \cite[\#2974]{Ashley1994}. 
	
	The corresponding generator tangles are \goi,\quad \goii,\quad and \goiii, respectively.
	
	\begin{figure}[H]
		\centering
		\begin{subfigure}[t]{0.33\textwidth}
			\centering
			\includegraphics[height=4.5cm]{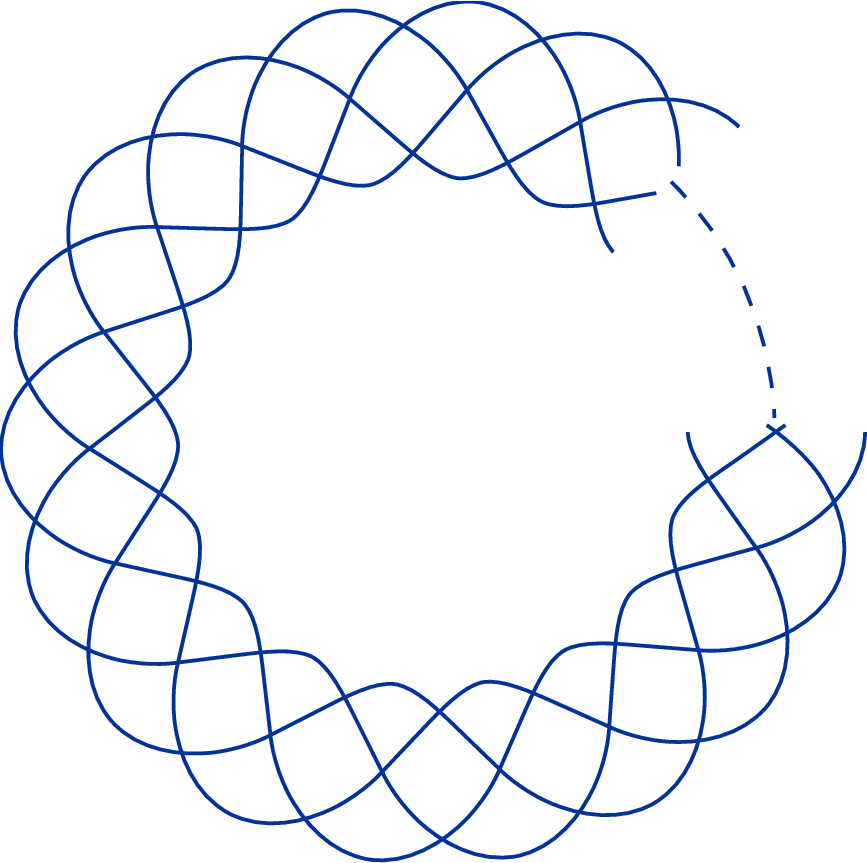}
			\caption{4-strand square Turk's head}
		\end{subfigure}%
		\begin{subfigure}[t]{0.33\textwidth}
			\centering
			\includegraphics[height=4.5cm]{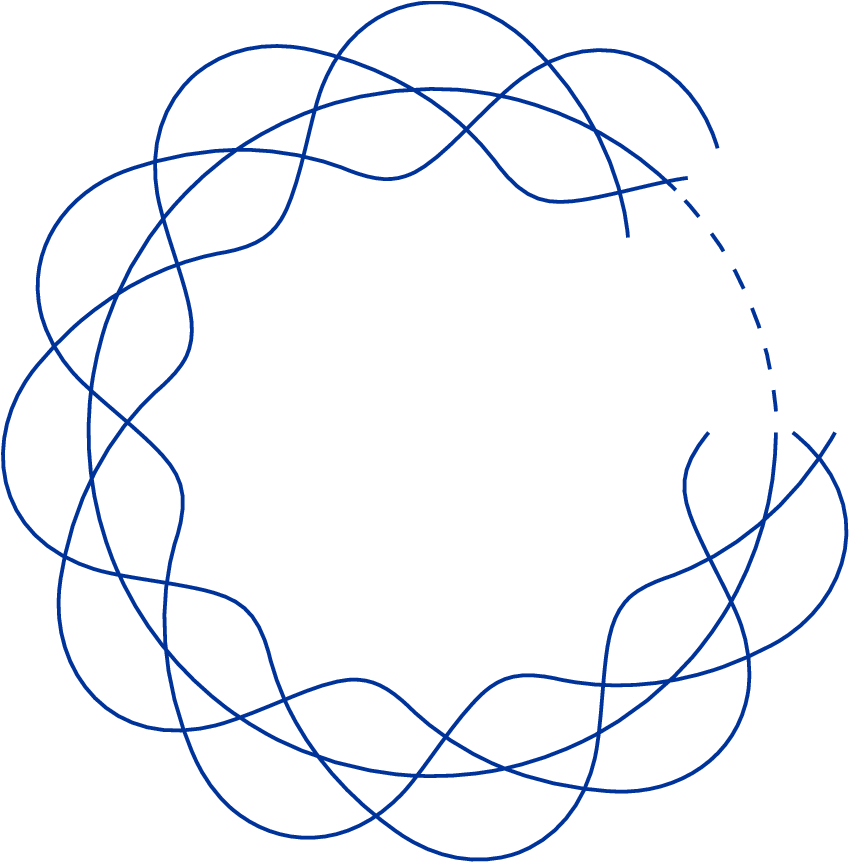}
			\caption{4-strand plat sinnet}
		\end{subfigure}%
		\begin{subfigure}[t]{0.33\textwidth}
			\centering
			\includegraphics[height=4.5cm]{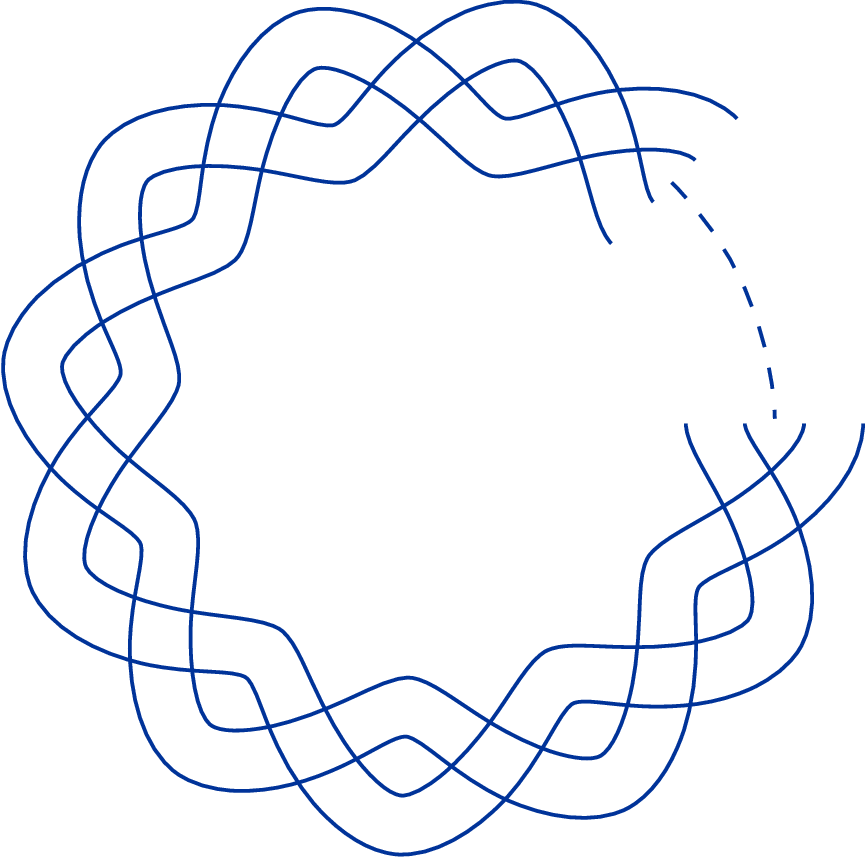}
			\caption{4-strand flat sinnet}
		\end{subfigure}%
		\caption{Closed 4-strand braided designs}
		\label{fig:4braid}
	\end{figure}
	
	In the following, the bracket polynomial is expressed in terms of the coefficients of its expansion for the first few values of $n$.
	\begin{enumerate}
		\item \textbf{4-strand square Turk's head knot}
		\begin{itemize}
			\item Bracket for the generator:
			
			\begin{align*}
				\hspace*{\dimexpr-\leftmargini-\leftmarginii}
				\brk{T} = \brk{\goi} & = \brk{\gix} + \brk{\gviii} + \brk{\gv}  + \brk{\giii}\\[2ex] 
				& = \brk{\gxi} + \brk{\giv}   + \brk{\gii} + \brk{\gi}.
			\end{align*}
			\item States matrix: 
			\[M(1,1,1,1,1,0,0,1,1,0,1,0,0,0).\]
			\item Bracket polynomial for ${T_n}$: 
			
			\begin{table}[H]
				\centering
				\makebox[0pt]{\resizebox{\linewidth}{!}{%
						\begin{tabular}{@{}c|lllllllllllll@{}}
							$n\setminus k$ & 0 & 1 & 2 & 3 & 4 & 5 & 6 & 7 & 8 & 9 & 10 & 11 & 12 \\	
							\midrule
							0 & 0 & 0    & 0     & 0     & 1\\
							1 & 0 & 1    & 3     & 3     & 1\\
							2 & 0 & 12   & 27    & 20    & 5\\
							3 & 0 & 75   & 188   & 169   & 67    & 12    & 1\\
							4 & 0 & 384  & 1148  & 1352  & 826   & 300   & 73    & 12    & 1\\
							5 & 0 & 1805 & 6417  & 9595  & 8101  & 4451  & 1755  & 518   & 110  & 15  & 1\\
							6 & 0 & 8100 & 33687 & 61566 & 66706 & 49172 & 26829 & 11310 & 3690 & 906 & 159 & 18 & 1\\
							& $\ldots$
				\end{tabular}}}
				\caption{Values of $\left[x^k\right]$ $\left<T_n\right>$ for small values of $n$}
				\label{tab:squareturkshead}
			\end{table}
		\end{itemize}
		
		\item \textbf{4-strand plat sinnet}
		\begin{itemize}
			\item Bracket for the generator:
			
			\begin{align*}
				\hspace*{\dimexpr-\leftmargini-\leftmarginii}
				\brk{T}=\left<\goii\right> & = \brk{\gi}  + \brk{\gii}       + (x+3)\brk{\giii} + \brk{\giv}  + \brk{\gv} \\[2ex]
				& + \brk{\gvi} + (x+3)\brk{\gvii} + \brk{\gxi}       + \brk{\gxii} + \brk{\gxiii}.
			\end{align*}
			
			\item States matrix:
			\[
			M(1, 1, x + 3, 1, 1, 1, x + 3, 0, 0, 0, 1, 1, 1, 0).
			\]
			\item Bracket polynomial for ${T_n}$:
			\begin{table}[H]
				\centering
				\makebox[0pt]{\resizebox{\linewidth}{!}{%
						\begin{tabular}{@{}c|lllllllllllll@{}}
							$n\setminus k$ & 0 & 1 & 2 & 3 & 4 & 5 & 6 & 7 & 8 & 9 & 10 & 11 & 12 \\	
							\midrule
							0 & 0 &  0     &  0      &  0      &  1 \\
							1 & 0 &  2     &  6      &  6      &  2 \\
							2 & 0 &  36    &  91     &  84     &  36     & 8      &  1 \\
							3 & 0 &  384   &  1148   &  1352   &  826    & 300    &  73    &  12    &  1 \\
							4 & 0 &  3528  &  12985  &  19904  &  16723  & 8608   &  2930  &  712   &  129  &  16   &  1\\
							5 & 0 &  30250 &  134910 &  260540 &  287276 & 201082 &  94435 &  31006 &  7465 &  1390 &  201 &  20 &  1\\
							& $\ldots$
				\end{tabular}}}
				\caption{Values of $\left[x^k\right]$ $\left<T_n\right>$ for small values of $n$}
				\label{tab:platsinnet}
			\end{table}
			
		\end{itemize}
		
		\item \textbf{4-strand flat sinnet}
		\begin{itemize}
			\item Bracket for the generator:
			\begingroup
			\allowdisplaybreaks
			\begin{align*}
				\hspace*{\dimexpr-\leftmargini-\leftmarginii}
				\brk{T} =	\left<\goiii\right> & = \brk{\gi} + \brk{\gii} + (x+4)\brk{\giii} + \brk{\giv} \\[2ex]
				& + \brk{\gv} + \brk{\gvi} + \brk{\gvii}      + \brk{\gviii} + \brk{\gix}\\[2ex]
				& + \brk{\gx} + \brk{\gxi} + \brk{\gxii}.
			\end{align*}
			\endgroup
			
			\item States matrix:
			\[M(1, 1, x + 4, 1, 1, 1, 1, 1, 1, 1, 1, 1, 0, 0).\]
			\item Bracket polynomial for ${T_n}$:
			\begin{table}[H]
				\centering
				\makebox[0pt]{\resizebox{\linewidth}{!}{%
						\begin{tabular}{@{}c|lllllllllllll@{}}
							$n\setminus k$ & 0 & 1 & 2 & 3 & 4 & 5 & 6 & 7 & 8 & 9 & 10 & 11 & 12\\	
							\midrule
							0 & 0 & 0     & 0     & 0      & 1\\
							1 & 0 & 2     & 6     & 6      & 2\\
							2 & 0 & 32    & 88    & 88     & 39     & 8      & 1\\
							3 & 0 & 294   & 1000  & 1364   & 961    & 378    & 86     & 12     & 1\\
							4 & 0 & 2304  & 9824  & 17904  & 18257  & 11424  & 4516   & 1120   & 170   & 16   & 1\\
							5 & 0 & 16810 & 87474 & 202410 & 274726 & 242402 & 145215 & 59686  & 16555 & 2960 & 317 & 20 & 1\\
							& $\ldots$
				\end{tabular}}}
				\caption{Values of $\left[x^k\right]$ $\left<T_n\right>$ for small values of $n$}
				\label{tab:flatsinnet}
			\end{table}
		\end{itemize}
	\end{enumerate}
	
	\begin{remark}\mbox{ }
		\begin{itemize}
			\item The closed 4-strand Turk's head knot is topologically equivalent to the closed circular Celtic knot $\CK_4^{2n}$. The component values of the bracket vector were computed in earlier work \cite{Ramaharo2025}; thus, by applying \eqref{eq:closuredefinitionD}, the bracket polynomial for the closure simplifies to
			\begin{equation}\label{eq:ClosedCelticclosedform}
				\begin{aligned}
					\left<\overline{T_n}\right> & = \left(x^2-1\right)\left((x+1)^n + \left(x+2-\sqrt{2 x+3}\right)^n + \left(x+2+\sqrt{2 x+3}\right)^n\right)\\
					& + \left(\frac{1}{2} \left(x^2+4 x+4-\sqrt{x^4+4 x^3+12 x^2+20 x+12}\right)\right)^n\\
					& + \left(\frac{1}{2} \left(x^2+4 x+4+\sqrt{x^4+4 x^3+12 x^2+20 x+12}\right)\right)^n + x^4 - 3x^2 + 1.
				\end{aligned}
			\end{equation}
			
			\item The sequence $ \left(\left[x\right] \left<T_n\right>\right)_n$ (Column  1) in Table~\ref{tab:squareturkshead} matches sequence \seqnum{A006235} in the OEIS \cite{Sloane2025}. Empirically, the initial values in Columns 1 of Tables~\ref{tab:platsinnet} and~\ref{tab:flatsinnet} match sequences \seqnum{A193127} and \seqnum{A212797}, respectively.
			
			\item 	For $n=1$, the bracket polynomials of the two 4-strand sinnets are identical. This can be seen in Figure~\ref{fig:statesexplication}: after smoothing specific crossings, the states of the two knots become identical up to isotopy on the sphere $\mathsf{S}^2$ (see Figure~\ref{fig:statesexplication}). Let $\brk{T_1}$ and $\brk{T_1'}$ denote the bracket polynomials of the respective 4-strand sinnets. Then,
			\begin{align*}
				\brk{T_1} = \brk{T_1'} & = (x+1)\brk{\protect\includegraphics[height=1.cm,valign=c]{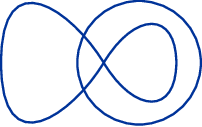}}\quad \mbox{by axiom \ref{itm:K2}}\\
				& = (x+1)(2x^3 + 4x^2 + 2x)\\
				& = 2x^4 + 6x^3 + 6x^2 + 2x.
			\end{align*} 
		\end{itemize}

		\begin{figure}[!ht]
			\centering
			\includegraphics[width=.8\linewidth]{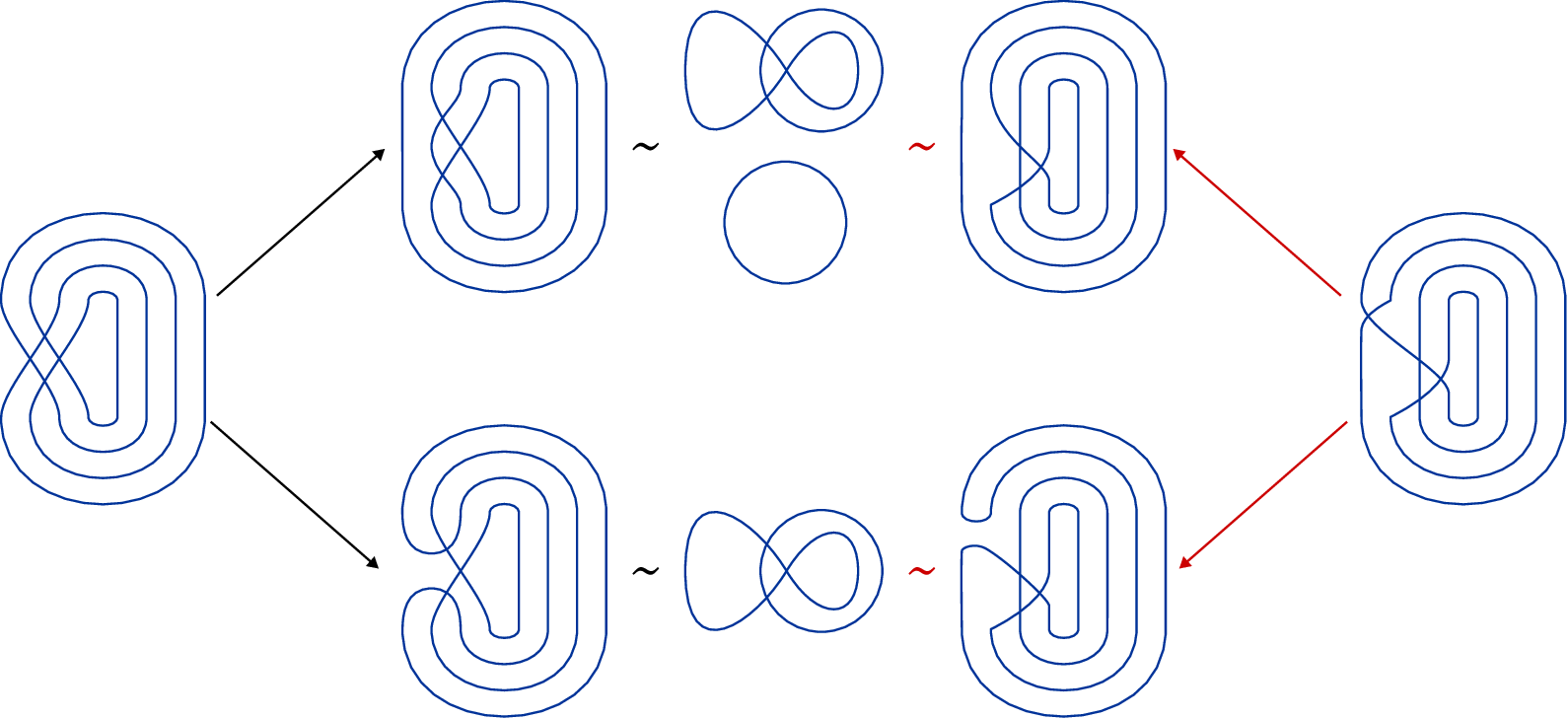}
			\caption{Isotopic equivalence of states for the 4-strand sinnets}
			\label{fig:statesexplication}
		\end{figure}
	\end{remark}


\begin{thebibliography}{99}
		\bibitem{Ashley1994}
		C.~W.~Ashley, \textit{The Ashley Book of Knots}, New York: Doubleday, 1944.
		
		\bibitem{KitovNikov2020}
		N.~V.~Kitov and M.~V.~Volkov, ``Identities of the Kauffman Monoid $\mathcal{K}_4$ and of the Jones Monoid $\mathcal{J}_4$'', In A.~Blass, P.~Cégielski, N.~Dershowitz, M.~Droste and B.~Finkbeiner (eds), \textit{Fields of Logic and Computation III. Lecture Notes in Computer Science} \textbf{12180} (2020), Springer, pp.~156--178.
		
		\bibitem{Ramaharo2025}
		F.~Ramaharo, The bracket polynomial of the Celtik link shadow $\CK_4^{2n}$, arXiv.org e-Print archive, 2025,  \url{https://arxiv.org/abs/2508.10410}
		
		\bibitem{Sloane2025} 
		N.~J.~A.~Sloane, \textit{The On-Line Encyclopedia of Integer Sequences}, published electronically at \url{http://oeis.org}, 2025.
		
	\end{thebibliography}
\end{document}